\documentclass[a4paper,12pt]{article}

\usepackage[utf8]{inputenc}

\usepackage{lmodern}
\usepackage[T1]{fontenc}

\usepackage[top=1in, bottom=1.25in, left=1.25in, right=1.25in]{geometry}
\usepackage{amsmath,amssymb,amsthm}

\usepackage{tikz}
\usepackage{xcolor}
\usepackage{array}
\usepackage{enumerate}
\usepackage{graphicx}
\usepackage{float}
\usepackage{caption}
\usepackage{subcaption}

\usepackage{amsrefs}
\usepackage[colorlinks=true,linkcolor=red,citecolor=blue,urlcolor=red]{hyperref}

\theoremstyle{definition}

\theoremstyle{remark}

\newtheorem*{rem}{Remark}

\theoremstyle{theorem}
\newtheorem{theorem}{Theorem}
\newtheorem{cor}[theorem]{Corollary}
\newtheorem{lem}[theorem]{Lemma}
\newtheorem{prop}[theorem]{Proposition}

\newtheorem*{theoremA}{Theorem A}
\newtheorem*{theoremB}{Theorem B}
\newtheorem*{theoremC}{Theorem C}

\newcommand{\overbar}[1]{\mkern 1.5mu\overline{\mkern-1.5mu#1\mkern-1.5mu}\mkern 1.5mu}

\newcommand{\divides}{\bigm|}

\author{Sel\c{c}uk Kayacan\thanks{This work is supported by the T{\"U}B\.{I}TAK 2214/A Grant Program: 1059B141401085.}}
\title{Connectivity of Intersection Graphs\\ of Finite Groups\thanks{I would like thank to Franz Lehner, Wilfried Imrich and Erg\"un Yaraneri for their helpful comments.}}
\date{}

\begin{document}

\maketitle

\small

\begin{center}
  Department of Mathematics, Istanbul Technical University,\\
  34469 Maslak, Istanbul, Turkey. \\
  {\it e-mail:} \href{mailto:skayacan@itu.edu.tr}{skayacan@itu.edu.tr}
\end{center}

\begin{abstract}
  The intersection graph of a group $G$ is an undirected graph without
  loops and multiple edges defined as follows: the vertex set is the
  set of all proper non-trivial subgroups of $G$, and there is an edge
  between two distinct vertices $H$ and $K$ if and only if
  $H\cap K \neq 1$ where $1$ denotes the trivial subgroup of $G$. In
  this paper, we classify finite solvable groups whose intersection
  graphs are not $2$-connected and finite nilpotent groups whose
  intersection graphs are not $3$-connected. Our methods are
  elementary.

  \smallskip
  \noindent 2010 {\it Mathematics Subject Classification.} Primary:
  20D99; Secondary: 20D15, 20D25, 05C25.

  \smallskip
  \noindent Keywords: Finite groups; subgroup; intersection graph;
  connectivity
\end{abstract}

\subsection*{Introduction}


Let $G$ be a group. We define the \emph{intersection graph} $\Gamma(G)$ of $G$ as the simple graph with vertex set the proper non-trivial subgroups of $G$ and there is an edge between two distinct vertices if and only if their intersection is non-trivial. It is not difficult to determine finite non-simple groups having a disconnected intersection graph:

\begin{theoremA}
  Let $G$ be a finite non-simple group. Then $\Gamma(G)$ is \emph{not} connected if and only if for some prime numbers $p$ and $q$ one of the following holds.
  \begin{enumerate}
  \item $G\cong\mathbb{Z}_p\times\mathbb{Z}_p$,\, or\,
    $G\cong\mathbb{Z}_p\times\mathbb{Z}_q$.
  \item $G\cong N\rtimes A$ where
    $N\cong \mathbb{Z}_p\times\dots\times\mathbb{Z}_p$,
    $A\cong\mathbb{Z}_q$, $N_G(A)=A$, and $N$ is a minimal
    normal subgroup of $G$.
  \end{enumerate}
\end{theoremA}

In \cite{Shen2010}, Shen proved this result and also showed that intersection graphs of (non-abelian) simple groups are connected, thereby completed the classification for all finite groups. Here we shall give a different proof for Theorem~A which is due to I. M. Isaacs. In an earlier work \cite{Lucido2003}, Lucido classified finite groups whose poset of proper non-trivial subgroups are connected. Obviously, $\Gamma(G)$ is connected if and only if the poset of proper non-trivial subgroups of $G$ is connected.

The aim of the present paper is to give a more detailed account of the ``connectivity'' of intersection graphs. For a connected graph $\Gamma$, a subset $\mathcal{S}$ of the vertex set $V(\Gamma)$ is said to be a \emph{separating set}, if removal of the vertices in $\mathcal{S}$ yields more than one component. We say $\Gamma$ is \emph{$k$-connected} if $|V(\Gamma)|>k$ and there is no separating set of cardinality $<k$. We define the \emph{connectivity} $\kappa(\Gamma)$ of $\Gamma$ as the greatest value of $k$ such that $\Gamma$ is $k$-connected. By convention, the connectivity of the complete graph $K_n$ on $n$ vertices is $n-1$. Hence, $1$-connected graphs form the class of connected graphs with at least two vertices. Clearly, $\Gamma$ is not connected if and only if $\kappa(\Gamma)=0$.
By abuse of notation, we denote the connectivity of the intersection graph of $G$ by $\kappa(G)$. For solvable groups we proved the following theorem. We denote the Frattini subgroup of $G$, i.e. the intersection of all maximal subgroups of $G$, by $\Phi(G)$.

\begin{theoremB}
  Let $G$ be a finite solvable group. Then $\kappa(G)<2$ if and only
  if for some prime numbers $p$ and $q$ one of the following holds.
  \begin{enumerate}
  \item $|G|=p^{\alpha}$ with $0\leq \alpha \leq 2$.
  \item $|G|=p^3$ and neither $G\cong Q_8$ nor
    $G\cong \mathbb{Z}_p\times\mathbb{Z}_p\times\mathbb{Z}_p$.
  \item $|G|=p^2q$ with a Sylow $p$-subgroup $P$ such that either
    \begin{enumerate}
    \item $P\cong\mathbb{Z}_{p^2}$, or
    \item $P\cong \mathbb{Z}_p\times\mathbb{Z}_p,\,P\triangleleft G,$ and there exists a non-normal subgroup of $G$ of order $p$.
    \end{enumerate}
  \item $G=PQ$ is a group of order $p^{\alpha}q$ ($\alpha\geq 3$) with
    $P$ being the normal Sylow $p$-subgroup of $G$ such that either
    \begin{enumerate}
    \item $P$ is elementary abelian, $Q$ acts on $P$ irreducibly, and
      the order of $N_G(Q)$ is at most $pq$, or
    \item $N:=\Phi(P)$ is elementary abelian, $Q$ acts on both $N$ and
      $P/N$ irreducibly, and either $N_G(Q)=Q$ or
      $N_G(Q)=NQ\cong\mathbb{Z}_p\times\mathbb{Z}_q$.
    \end{enumerate}
  \end{enumerate}
  In particular, any solvable group whose order is divisible by at least three distinct primes is $2$-connected.
\end{theoremB}

Intuitively, intersection graphs should be highly connected graphs and if there are some examples of such graphs with `low' connectivity, they must be exceptional. By Menger's Theorem (see \cite[Theorem 3.3.6]{Diestel2005}), a graph is $k$-connected if and only if it contains $k$ independent paths between any two vertices. Hence, if $\Gamma(G)$ is $3$-connected, there must exist sufficiently many vertices in the intersection graph forming at least three independent paths between any pair of vertices. However, claiming the existence of those subgroups and also verifying that they intersect non-trivially sufficiently many times seems to be a fairly complicated problem for the class of solvable groups. For nilpotent groups we obtain the following theorem.

\begin{theoremC}
  Let $G$ be a finite nilpotent group. Then $\kappa(G)<3$ if and only if for some prime numbers $p,q$, and $r$ one of the following holds.
  \begin{enumerate}
  \item $|G|=p^{\alpha}$ $(0\leq\alpha\leq 3)$ and neither
    $G\cong Q_8$ nor
    $G\cong \mathbb{Z}_p\times\mathbb{Z}_p\times\mathbb{Z}_p$.
  \item $G$ is a group of order $p^4$ such that
    \begin{enumerate}
    \item $G\cong \mathbb{Z}_{p^4}$, or
    \item $\Phi(G)\cong\mathbb{Z}_{p^2}$ except $G\cong Q_{16}$, or
    \item
      $\Phi(G)\cong\mathbb{Z}_p\times\mathbb{Z}_p$ and $Z(G)<\Phi(G)$
      except
      \begin{enumerate}
      \item $G\cong \langle a,b,c\mid a^9=b^3=1, ab=ba, a^3=c^3,
      bcb^{-1}=c^4, aca^{-1}=cb^{-1} \rangle $, and 
      \item $G\cong \langle a,b,c\mid a^{p^2}=b^p=c^p=1, bc=cb, bab^{-1}=a^{p+1}, cac^{-1}=ab \rangle$ for $p>3$. 
      \end{enumerate}      
    \end{enumerate}
  \item $G\cong\mathbb{Z}_{p^3q}$, $G\cong\mathbb{Z}_{p^2q}$,
    $G\cong(\mathbb{Z}_p\times\mathbb{Z}_p)\times\mathbb{Z}_q$, or
    $G\cong\mathbb{Z}_{pqr}$.
  \end{enumerate}
 Moreover, any solvable group whose order is divisible by at least four
  distinct primes is $3$-connected.
\end{theoremC}

Let $\Gamma$ be a simple graph with vertex set $V(\Gamma)$. A sequence $\gamma=(v_0,v_1,\dots,v_k)$ of vertices is a \emph{path} of length $k$ between $v_0$ and $v_k$, if each consecutive pair of vertices are adjacent in $\Gamma$. We call two or more paths with the same end points \emph{internally independent} provided that none of them have a common inner vertex with another. For brevity, we usually omit `internally' and say simply `independent paths'. We shall note that occasionally we attribute the graph theoretical properties of intersection graphs directly to the group itself. For example, instead of ``$\Gamma(G)$ is $k$-connected'' we may simply say ``$G$ is $k$-connected''.

\subsection*{Preliminaries}

Let $V(G)$ be the set of proper non-trivial subgroups of $G$. This vertex set $V(G)$ (of $\Gamma(G)$) naturally carries a poset structure under set inclusion and its minimal elements are the minimal subgroups of $G$. A subset $\mathcal{S}$ of $V(G)$ is \emph{upward closed} if whenever $H\in \mathcal{S}$ and $H\leq K$, then also $K\in\mathcal{S}$.

\begin{prop}\label{prop:upward}
  For a finite group $G$ with $|V(G)|>k$ the following statements are
  equivalent:
  \begin{enumerate}[(i)]
  \item $\Gamma(G)$ is $k$-connected.
  \item There is no upward closed separating set $\mathcal{S}$ of
    $\Gamma(G)$ with $|\mathcal{S}|<k$.
  \item There are at least $k$ independent paths in $\Gamma(G)$
    between any pair of minimal subgroups.
  \end{enumerate}
\end{prop}

\begin{proof}
  (i)\!$\iff$\!(ii): By definition a graph is $k$-connected if and
  only if there is no separating set of cardinality $<k$. Thus, all we
  need to do is to show that any minimal separating set for
  $\Gamma(G)$ is upward closed (except, if $\Gamma(G)$ is a complete
  graph). Take a vertex $S\in\mathcal{S}$ where $\mathcal{S}$ is a
  minimal separating set. By the minimality of $\mathcal{S}$, for any
  two vertices $H,K\in V(G)\setminus\mathcal{S}$ there is a path
  $\gamma=(H,\dots,K)$ traversing only the points in
  $(V(G)\setminus\mathcal{S})\cup\{S\}$. Suppose that $H$ and $K$
  belong to different components (obtained after removing all the
  vertices in $\mathcal{S}$). So $\gamma$ necessarily visits $S$,
  i.e. $\gamma=(H,\dots,S,\dots,K)$. If $\overbar{S}\in V(G)$ and
  $S<\overbar{S}$, then $\bar{\gamma}=(H,\dots,\overbar{S},\dots,K)$
  is also a path from $H$ to $K$ and therefore
  $\overbar{S}\in\mathcal{S}$. Since $S$ was chosen arbitrarily,
  $\mathcal{S}$ is upward closed.

  (i)\!$\iff$\!(iii): Menger's Theorem states that a graph is
  $k$-connected if and only if it contains $k$ independent paths
  between any two vertices. Therefore, it is enough to show that
  existence of $k$ independent paths between any pair of minimal
  subgroups implies the existence of $k$ independent paths between any
  pair of subgroups in $V(G)$. If there exists a unique minimal
  subgroup of $G$, then $\Gamma(G)$ is a complete graph on more than
  $k$ vertices, thus it is $k$-connected. Suppose that there are more
  than one minimal subgroups of $G$. Let $X,Y\in V(G)$ be two distinct
  vertices and $A,B$ be two minimal subgroups with
  $\gamma_i=(A,A_i,\dots,B_i,B)$, $1\leq i\leq k$, are independent
  paths between them. Suppose that neither $X$ nor $Y$ are minimal
  subgroups. There are two cases that may occur:

  \emph{Case I:} $X$ and $Y$ contain a common minimal subgroup, say
  $A$. Then $\bar{\gamma}_i:=(X,A_i,Y)$ are independent paths provided
  that no coincidence occurs. If $X$ coincides with, say $A_1$, then
  replace $\bar{\gamma}_1$ with $(X,Y)$. If, in addition, $Y$
  coincides with, say $A_2$, then substitute
  $(X,\dots,B_1,B_2,\ldots,Y)$ for $\bar{\gamma}_2$.

  \emph{Case II:} $X$ and $Y$ contain distinct minimal subgroups, say
  $A$ and $B$ respectively. In this case, we may simply take
  $\bar{\gamma}_i=(X,A_i,B_i,Y)$ as independent paths between $X$ and
  $Y$. If $X$ or $Y$ coincides with some inner vertex, we may simply
  shorten the path accordingly.

  Finally, it is easy to see that above arguments can still be applied
  with minor changes if one of $X$ and $Y$ is a minimal subgroup.
\end{proof}

Obviously, if a graph is $k$-connected, then the degree (valency) of any vertex is at least $k$. In view of Proposition~\ref{prop:upward}\,(iii) we make the following convention: For a finite group $G$, we say \begin{quote}
  ``$G$ satisfies the $k$-valency condition''
\end{quote} provided that any minimal subgroup of $G$ is strictly contained in at least $k$ proper subgroups.

A vertex $v$ of a connected graph $\Gamma$ is called a \emph{cut-vertex}, if removing $v$ from $\Gamma$ renders a disconnected graph, i.e. if $\{v\}$ is a separating set for $\Gamma$. For the complete graph $K_2$, we shall regard any of its two vertices as a cut-vertex. (This is not the standard convention.) Hence, $\kappa(\Gamma)=1$ if and only if there exists a cut-vertex of $\Gamma$. 

\begin{lem}\label{lem:abelian}
  Let $G$ be a finite abelian group. Then there exists a cut-vertex
  of $\Gamma(G)$ if and only if $G$ is isomorphic to one of the
  following groups  $$\mathbb{Z}_{p^3},\qquad\mathbb{Z}_{p^2}\times\mathbb{Z}_p,\qquad\mathbb{Z}_{p^2}\times\mathbb{Z}_q$$
  for some prime numbers $p$ and $q$.
\end{lem}

\begin{proof}
  Let $G$ be a finite abelian group such that there is a cut-vertex
  $M$ in $\Gamma(G)$. By Proposition~\ref{prop:upward}\,(ii), $M$ can
  be taken as a maximal subgroup of $G$. Actually, $M$ must be a
  maximal subgroup unless $\Gamma(G)$ is a complete graph. Suppose
  that $\Gamma(G)\cong K_n$. Obviously it must be the complete graph
  on two vertices. In other words, $G$ has a unique maximal subgroup
  and a unique minimal subgroup. This is possible only if
  $G\cong\mathbb{Z}_{p^2}$ for some prime number $p$. (Observe that a
  finite group has a unique maximal subgroup if and only if it is
  isomorphic to a cyclic group of prime power order.)

  Next, suppose that $\Gamma(G)$ is not complete. Clearly, there are
  more than one minimal subgroups, i.e. $G$ is not cyclic of prime power order. Let $A$ be a minimal subgroup of $M$. As any subgroup contains
  a minimal subgroup, any component of the graph obtained by removing the vertex $M$ and all the incident edges to $M$ from $\Gamma(G)$ contains at least one minimal subgroup. Let $B$ be a minimal subgroup which is not in the same
  component with $A$. Since $(A,AB,B)$ is a path between them,
  $M=AB$. Moreover, observe that for either $A$ or $B$, say $A$, the subgroup $AB$ must be the unique proper subgroup containing $A$ strictly. Otherwise, there would be a path $(A,X,Y,B)$ between $A$ and $B$, where $X$ and $Y$ are some maximal subgroups different from $AB$ that are containing $A$ and $B$ respectively. It can be easily seen that $M=AB$ is a cut-vertex of the abelian group $G$ if and only if $G$ is isomorphic to either
  $\mathbb{Z}_{p^2}\times\mathbb{Z}_p$, or
  $\mathbb{Z}_{p^2}\times\mathbb{Z}_q$ for some prime numbers $p$ and
  $q$. 
\end{proof}

In our context, it is important to know when two minimal subgroups generate a preferably small proper subgroup. Accordingly, it is easier to describe the connectivity of groups with \emph{many} normal subgroups such as $p$-groups. On the other hand, it is known that any simple group can be generated by two elements. Let us recapitulate some basic group theoretical facts that are essential for our arguments.

The \emph{Frattini subgroup} $\Phi(G)$ of a group $G$ is the intersection of all maximal subgroups of $G$. It is well-known that the quotient of a finite $p$-group by its Frattini subgroup is elementary abelian. Moreover, $\Phi(G)$ is the minimal subgroup with this property. Therefore, $\Phi(G)=1$ if and only if $G$ is elementary abelian (see \cite[Theorem~5.1.3]{Gorenstein1980}). Notice that $\Phi(G)$ is a normal (even characteristic) subgroup of $G$. 

The $p$-core $O_p(G)$ of a finite group $G$ is the intersection of all Sylow $p$-subgroups of $G$. Like $\Phi(G)$ it is a characteristic subgroup; actually, it is the unique largest normal $p$-subgroup of $G$. In a finite solvable group $G$, the factors of every chief series are elementary abelian of prime power order. In particular, every minimal normal subgroup of $G$ is elementary abelian (see \cite[Theorem~2.4.2]{Gorenstein1980}). Hence, for a non-trivial solvable group $G$, there exists a prime $p\divides  |G|$ such that $O_p(G)$ is non-trivial.

A finite group $G$ is called \emph{supersolvable} if it possesses a normal series in which every factor group is cyclic of prime order. If a finite group is supersolvable, then every maximal subgroup is of prime index (see \cite[p.~85]{Isaacs2008}); and therefore, any maximal chain of subgroups have the same length. Let $G$ be a group of order $p_1^{\alpha_1}p_2^{\alpha_2}\dots p_k^{\alpha_k}$ where $p_i$ ($1\leq i\leq k$) are distinct prime numbers. We define the \emph{order length} of $G$ as $\ell(G):=\sum_{i=1}^k\alpha_i$. Clearly, for a supersolvable group $G$, the order length $\ell(G)$ is equal to the length of a maximal chain. Supersolvable groups form a class between the class of nilpotent groups and the class of solvable groups.

We close this section by presenting another structural result. Observe that the intersection graph of the trivial group $1$ as well as the intersection graph of $\mathbb{Z}_p$ ($p$ is a prime) are empty graphs. However, we set $|V(1)|=-1$ and $|V(\mathbb{Z}_p)|=0$ to make the statement of the following Proposition easier. Moreover, we adopt the following convention $$ \kappa(1)=-2,\qquad \kappa(\mathbb{Z}_p)=-1,\qquad \kappa(\mathbb{Z}_{p^2})=\kappa(K_1)=0. $$
Notice that this is in conformity with the our previous convention that $\kappa(K_n)=n-1$. 

\begin{prop}
  Let $G$ be a finite group and $N$ be a normal subgroup of $G$. If
  $G/N$ is $k$-connected, then $G$ is $(k+x-1)$-connected where $x$ is
  the length of the series $$1<N_1<N_2<\dots<N_x=N $$ such that
  $N_i\triangleleft G$ for each $1\leq i\leq x$. In particular,
  $\kappa(G/N)\leq \kappa(G)$.
\end{prop}

\begin{proof}
  Let $G$ and $N$ be as in the hypothesis of the Proposition. Let $A$
  and $B$ be two minimal subgroups of $G$. If $\kappa(G/N)=-2$, then
  there is a normal series $$ 1<N_1<N_2<\dots<N_x=G, $$ and we may
  easily form $x-2$ independent paths $\gamma_i=(A,N_iA,N_iB,B)$,
  $1\leq i\leq x-2$, between $A$ and $B$. (In case of a possible
  coincidence of the vertices we can safely shorten the paths.) A
  similar argument shows that we may construct $x-1$ independent paths
  if $\kappa(G/N)=-1$.

  Next suppose that $\kappa(G/N)\geq 0$, i.e. $|V(G/N)|\geq 1$. By the
  Correspondence Theorem there is a bijection between the subgroups of
  the quotient group $G/N$ and the subgroups of $G$ that are
  containing $N$. Observe that $NA$ and $NB$ correspond to some
  subgroups of $G/N$ that are either trivial or minimal. Then, as
  $G/N$ is $k$-connected by the assumption, we may construct at least
  $k$ additional independent paths $\gamma_j=(A,\dots,B)$,
  $x\leq j\leq k+x-1$, such that the inner vertices represents some proper subgroups of $G$ containing $N$.
\end{proof}

\begin{cor}\label{cor:order}
  Let $G$ be a supersolvable group with $\ell:=\ell(G)$. Then
  $\kappa(G)\geq \ell-3$. In particular, all $p$-groups of order
  $> p^{\alpha}$ are $(\alpha - 2)$-connected. \qed
\end{cor}

\subsection*{Non-simple groups}

\begin{proof}[Proof of Theorem A]
Let $G$ be a finite non-simple group and $N$ be a minimal normal subgroup of $G$. Suppose that $\Gamma(G)$ is not connected. Let $A$ be a subgroup of $G$ which does not lie in the component of $N$ in $\Gamma(G)$. Then $NA=G$, as otherwise, $(N,NA,A)$ would be a path between $N$ and $A$. Also $N\cap A=1$, as otherwise, $N$ and $A$ would be linked via the subgroup $N\cap A$. Therefore $|G:N|=|A|$. Since this equality holds for every subgroup that does not lie in the component containing $N$, it holds also for any non-trivial subgroup of $A$. As a consequence $|G:N|=|A|=q$ is a prime number. Moreover, $A$ is a maximal subgroup of $G$. To see this, suppose that there exists a proper subgroup $B$ containing $A$. Since $B$ does not lie in the same component with $N$, we have $|B|=q$, i.e. $B$ coincides with $A$.

Let $Q$ be a Sylow $q$-subgroup of $G$ containing $A$. Since $A$ is a maximal subgroup, either $G=Q$ or $A=Q$. In the first case since $N$ is a minimal normal subgroup and $G$ is a $q$-group, the order of $N$ is $q$. As $N$ and $A$ are distinct subgroups of same order $q$ and as $G=NA$, we see that $G\cong \mathbb{Z}_q\times\mathbb{Z}_q$. Clearly, $\Gamma(\mathbb{Z}_q\times\mathbb{Z}_q)$ is not connected.


In the latter case since $G$ is not a $q$-group and since $G=NA$, there must be a prime $p$ dividing $|N|$ and different from $q$. We want to show that $N$ is a $p$-group. Suppose contrarily that $N$ is not a $p$-group. Let $P$ be a Sylow $p$-subgroup of $N$ and $T=N_G(P)$. (Notice that $G\neq T$, as $N$ is a minimal normal subgroup.) By the Frattini argument $G=NT$ which, in turn, implies that $q\divides |T|$. Since $A$ is a Sylow $q$-subgroup, some conjugate of $T$ contains $A$. However, this contradicts with the maximality of $A$. Therefore, $N$ is a $p$-subgroup. Further, $N$ must be an elementary abelian subgroup since it is a minimal normal subgroup.



Consider the normalizer $N_G(A)$. Since $A$ is a maximal subgroup, there are two possibilities. If $A$ is a normal subgroup of $G$, then $A$ centralizes $N$; hence, $|N|=p$ and $G\cong\mathbb{Z}_p\times\mathbb{Z}_q$. Clearly, $\Gamma(\mathbb{Z}_p\times\mathbb{Z}_q)$ is not connected. And if $A$ is self-normalizing, $G$ is a group described as in the second part of Theorem~A.
To conclude the proof it is enough to show that $\Gamma(G)$ is not connected in such a case.

Let $H$ be a proper non-trivial subgroup of $G$. We want to show that $H$ is either a subgroup of the unique (normal) Sylow $p$-subgroup $N$ of $G$ or it is a Sylow $q$-subgroup. Obviously, $\Gamma(G)$ is not connected if this is the case. Suppose contrarily, $H$ is neither a $p$-subgroup nor a $q$-subgroup. Then $q\divides |H|$ as $|G|=|NA|=p^{\alpha}q$ for some integer $\alpha\geq 1$. Hence, $H$ contains a conjugate of $A$ and we may suppose that $H$ contains $A$ by replacing $H$ with some conjugate of it if necessary. Then $NH=G$ and it follows that $N\cap H\triangleleft G$. (Notice that $N\cap H$ is normalized by $N$ as $N$ is an abelian subgroup and $N\cap H$ is normalized by $H$ as $N$ is a normal subgroup.) Since $N$ is a minimal normal subgroup, either $N\cap H=1$ or $N\cap H=N$ yielding either $|H|=q$ or $H=G$. However, this contradicts with the assumption that $H$ is a proper subgroup which is not a $q$-subgroup. 

\end{proof}

Notice that for a finite non-simple group $G$, the connectivity of $G$ is $1$ if and only if $G$ satisfies the $1$-valency condition.

\subsection*{Solvable groups}

\begin{lem}\label{lem:valency2}
  Let $G$ be a finite solvable group. Then $\kappa(G)\geq 2$ if and only
  if $G$ satisfies the $2$-valency condition.
\end{lem}

\begin{proof}
  Sufficiency is obvious. Let $G$ be a finite solvable group
  satisfying the $2$-valency condition. We want to show that there
  exist at least two independent paths between any pair of minimal
  subgroups $A_1$ and $A_2$. If $\langle A_1,A_2\rangle$ is a second
  maximal subgroup, then clearly there are two independent paths
  between them. Thus, for the rest we assume $\langle A_1,A_2\rangle$
  is either $G$ or a maximal subgroup. Let $M$ be a maximal subgroup
  of prime index and $N$ be a minimal normal subgroup. Notice that
  since $G$ is solvable, there exist a subgroup of prime index and
  minimal normal subgroups are elementary abelian. Further, let $A_1<H_1,K_1$
  and $A_2<H_2,K_2$ such that $NA_1\neq H_1$ and $NA_2\neq H_2$.

  \emph{Case I:} Suppose that $N$ is of prime index in $G$ and take
  $M=N$.

  \emph{Case I (a):} $A_1,A_2<M=N$. Obviously $(A_1,M,A_2)$ is a path
  and $\langle A_1,A_2\rangle=M\cong\mathbb{Z}_p\times\mathbb{Z}_p$.
  And the order of $G$ is either $p^3$ or $p^2q$. By the product
  formula, $(A_1,H_1,H_2,A_2)$ is also a path and independent from the
  first.

  \emph{Case I (b):} $M$ and $A_1$ are distinct $p$-groups. Then $G$
  is also a $p$-group and in turn $|G|=p^2$ since $M$ must be a cyclic
  group of prime order. However, intersection graph of a group of
  order $p^2$ or $pq$ consists of isolated vertices and $G$ does not
  satisfy $2$-valency condition in that case.

  \emph{Case I (c):} $M$ is an elementary abelian $p$-group of rank
  $n\geq 2$ and $A_1\cong \mathbb{Z}_q$. In particular
  $|G|=p^nq$. Observe that $A_1\ntriangleleft G$, as otherwise, $G$
  would be an abelian group contradicting the fact that $M$ is a
  minimal normal subgroup. Moreover, $O_p(H_1)$ and
  $O_p(K_1)$ are trivial (again this is because $M$ is a
  minimal normal subgroup) and this in turn implies
  $H_1,K_1\trianglelefteq N_G(A_1)<G$. (Notice that this
  implies $n\geq 3$). Hence, we may assume
  $A_1<H_1<K_1=N_G(A_1)$. If $A_2<M$, then we have the two
  independent paths $(A_1,H_1,M,A_2)$ and $(A_1,K_1,T,A_2)$ where $T$
  is a subgroup of order $p^{n-1}$ containing $A_2$. And if $A_2$ is a
  conjugate of $A_1$, then $(A_1,H_1,M,H_2,A_2)$ and
  $(A_1,K_1,T,K_2,A_2)$ are two independent paths between $A_1$ and
  $A_2$ where $H_2<K_2$.

  \emph{Case II:} Suppose that $[G:N]$ is not prime. Then
  $NA_1\neq G$, $NA_2\neq G$. If one of $NA_1$ and $NA_2$ coincides
  with $M$, say $NA_1$, then we may take $(A_1,H_1,NA_2,A_2)$ and
  $(A_1,K_1,M,H_2,A_2)$ as independent paths. If both $NA_1$ and
  $NA_2$ coincides with $M$, then we may take $(A_1,M,A_2)$ and
  $(A_1,H_1,N,H_2,A_2)$. Finally, if $NA_1\neq M$ and $NA_2\neq M$,
  then $(A_1,NA_1,NA_2,A_2)$ and $(A_1,H_1,M,H_2,A_2)$ are two
  independent paths between $A_1$ and $A_2$.
\end{proof}

\begin{lem}\label{lem:pgroup2}
  Let $G$ be a finite $p$-group. Then $\kappa(G)<2$ if and only if
  \begin{enumerate}
  \item $|G|=p^{\alpha}, \; 0\leq\alpha\leq 2,$
  \item
    $|G|=p^3\; \text{ except }\; G\cong Q_8\, \text{ or }\, G\cong
    \mathbb{Z}_p\times\mathbb{Z}_p\times\mathbb{Z}_p$
  \end{enumerate}
  In particular, all $p$-groups of order $>p^3$ are $2$-connected.
\end{lem}

\begin{proof}
  By Lemma~\ref{lem:valency2}, all we need to do is to determine
  $p$-groups which do not satisfy the $2$-valency condition. Clearly,
  intersection graph of a group of order $p^{\alpha}$,
  $0\leq \alpha \leq 2$, is either empty graph or consists of isolated
  vertices. Hence $2$-valency condition does not hold for those
  groups.

  Suppose that $|G|=p^3$. If $G$ has a unique maximal subgroup, then
  $G\cong\mathbb{Z}_{p^3}$ and $\Gamma(G)\cong K_2$. So it is not
  $2$-connected in this case. If $G$ has more than one maximal
  subgroup and $\Phi:=\Phi(G)$ is non-trivial, then either $G$ has a
  unique minimal subgroup (which is $\Phi$) or there are minimal
  subgroups different from $\Phi$. In the first case, $G\cong Q_8$ and
  $\Gamma(Q_8)\cong K_4$. That is, $G$ is $3$-connected. In the latter
  case, $\Phi A$ is a maximal subgroup of $G$ and it is the unique
  subgroup of order $p^2$ containing $A$, as all the maximal subgroups
  contain $\Phi$. If $\Phi$ is trivial, then $G$ is elementary abelian
  and by the Correspondence Theorem any minimal (normal) subgroup is
  contained in $p+1$ maximal subgroups. Therefore, $G$ is $2$-connected in
  this case, as the $2$-valency condition holds.

  Suppose that $|G|=p^{\alpha}$, $\alpha>3$. Then any minimal subgroup
  of $G$ is contained in a subgroup of order $p^2$ and by a subgroup
  of order $p^3$. Hence $G$ satisfies the $2$-valency condition.
\end{proof}

\begin{lem}\label{lem:p2q}
  Let $G$ be a group of order $p^2q$ with a Sylow $p$-subgroup $P$.
  Then $\kappa(G)<2$ if and only if one of the following holds.
  \begin{enumerate}
  \item $P\cong \mathbb{Z}_{p^2}$.
  \item $P\cong \mathbb{Z}_p\times\mathbb{Z}_p,\, P\triangleleft G,$ and there exists a
    non-normal subgroup of $G$ of order $p$.
  \end{enumerate}
\end{lem}

\begin{proof}
  Let $G$ be a group of order $p^2q$ with a Sylow $p$-subgroup $P$ and
  a Sylow $q$-subgroup $Q$. It is a well-known fact that a group of order $p^2q$
must have a normal Sylow subgroup. Thus, either $P$ or $Q$ is a normal subgroup of $G$.
  
  \emph{Case I:} $P\cong\mathbb{Z}_{p^2}$. If $P\triangleleft G$, then
  $G$ has a unique subgroup of order $p$. However, this implies any
  $q$-subgroup is contained in one and only one subgroup (of order
  $pq$). Assume that $P\ntriangleleft G$. Then $Q$ is a normal
  subgroup of $G$ and there exists a normal subgroup
  $M\cong\mathbb{Z}_q\rtimes\mathbb{Z}_p$ containing all subgroups of
  order $p$. Those two facts imply that $M$ is the unique subgroup
  containing $Q$.

  \emph{Case II:} $P\cong\mathbb{Z}_p\times\mathbb{Z}_p$.

  \emph{Case II (a):} $Q\triangleleft G$. Clearly, $2$-valency condition holds in this case. 

  \emph{Case II (b):} $Q\ntriangleleft G$. As was remarked earlier the Sylow $p$-subgroup must be a normal subgroup in this case. First, we shall observe
  that either there is no normal subgroup of $G$ of order $p$ or there
  are more than one. As $Q\ntriangleleft G$, the index of
  $N_G(Q)$ is either $p$ or $p^2$ and this implies
  $q\divides  p-1$ or $q\divides p+1$. Consider the action of $Q$ on the
  subgroups of $P$ by conjugation. Since the length of an orbit is
  either $1$ or $q$, the number of fixed points (the number of normal
  subgroups of order $p$) may be congruent $0$ or $2$ modulo $q$. Next,
  we determine the groups in which $Q$ is contained in at most one
  subgroup. If $Q$ acts on $P$ irreducibly (without fixed points) and
  $Q$ is contained in subgroup $M$ of order $pq$, then
  $M\cong\mathbb{Z}_q\rtimes\mathbb{Z}_p$ and it normalizes
  $Q$. Moreover, it is the unique subgroup containing $Q$ as
  $Q\ntriangleleft G$. If there are distinct normal subgroups $U$ and
  $V$ of order $p$, then clearly $UQ$ and $VQ$ are distinct
  subgroups containing $Q$. Finally, we determine the groups in which
  a (non-normal) subgroup $T$ of order $p$ is not contained in a
  subgroup of order $pq$. As we have seen that groups in which $Q$
  acts on $P$ irreducibly does not satisfy $2$-valency condition, we
  further assume that there exist two normal subgroups $U$ and $V$ of
  order $p$. Suppose that $T$ is contained in a subgroup $M$ of order
  $pq$. Then as $T\ntriangleleft G$, we have
  $M\cong\mathbb{Z}_q\rtimes\mathbb{Z}_p$. On the other hand, both
  $UQ$ and $VQ$ cannot be isomorphic to
  $\mathbb{Z}_p\times\mathbb{Z}_q$, as otherwise, $Q<Z(G)$
  implying $G$ is abelian. That is, one of $UQ$ and $VQ$ is isomorphic
  to $\mathbb{Z}_p\rtimes\mathbb{Z}_q$ which is impossible. Therefore,
  $T$ is not contained in any subgroup of order $pq$.





\end{proof}

As is seen by Lemma~\ref{lem:p2q}, many of the groups of order $p^2q$ do not satisfy $2$-valency condition. Compare it with the following result.

\begin{lem}\label{lem:p2q-2}
  Let $G$ be a group of order $p^2q$. Then $G$ is $3$-connected if and
  only if
	 $$G\cong \langle a,b,c\mid a^p=b^p=c^q=1, ab=ba, cac^{-1}=a^{\lambda}, cbc^{-1}=b^{\lambda} \rangle $$ 
         where $q\divides  p-1$ and $\lambda>1$ is any integer such
         that $\lambda^q\equiv 1 \pmod{p}$.
\end{lem}

\begin{proof}
  Let $G$ be a $3$-connected group of order $p^2q$ and let $Q$ be a
  $q$-subgroup of $G$. Take a minimal $p$-subgroup $U$ of $G$ and let
  $P$ be a Sylow $p$-subgroup containing $U$.

  \emph{Case I:} $Q\triangleleft G$. Clearly, $QU$ is the unique
  subgroup of order $pq$ containing $U$. Moreover, there exist at
  least two distinct Sylow $p$-subgroups containing $U$, as otherwise,
  $3$-valency condition does not hold. This, in turn, implies that
  $N_G(U)=G$. Suppose that Sylow $p$-subgroups are
  cyclic. Then, $U$ is the unique subgroup of order $p$ and $QU$ is
  the unique subgroup containing $Q$. Again $3$-valency condition
  cannot be satisfied. Now, suppose that Sylow $p$-subgroups are
  elementary abelian. Since $U$ is a normal subgroup of $G$ and since
  this must be the case for any minimal $p$-subgroup, $P$ is also a
  normal subgroup of $G$ which is a contradiction.

  \emph{Case II:} $Q\ntriangleleft G$. Since $G$ is a solvable group,
  the $p$-core $O_p(G)$ is a non-trivial normal subgroup of
  $G$. Thus, we shall consider following two sub-cases.

  \emph{Case II (a):} $P\ntriangleleft G, U\triangleleft G$. Suppose
  that Sylow $p$-subgroups are cyclic. As in Case I, $PQ$ is the
  unique subgroup containing $Q$ and this case can be discarded. Now,
  suppose that Sylow $p$-subgroups are elementary abelian. However, by
  the proof of Lemma~\ref{lem:p2q} we know that no such group exists.


  \emph{Case II (b):} $P\triangleleft G$. As in previous cases, $P$
  cannot be a cyclic subgroup. Thus
  $P\cong\mathbb{Z}_p\times\mathbb{Z}_p$. We claim that
  $N_G(Q)=Q$. Assume contrarily that $N_G(Q)$ is a
  group of order $pq$. Then $N_G(Q)$ is self-normalizing and
  by the product formula any two distinct conjugates of it intersect at
  $Q$; however, $Q$ is normal in both of them which is a
  contradiction. Therefore, any subgroup of order $pq$ must be
  isomorphic to $\mathbb{Z}_p\rtimes\mathbb{Z}_q$ and since $G$ is
  $3$-connected there must exist at least three such subgroups
  containing $Q$. To write a presentation for $G$, let $a,b,c$ be
  three elements generating $G$ such that $a,b$ are of order $p$ and
  $c$ is of order $q$. Moreover, we may suppose that $c$ normalizes
  $\langle a\rangle$, $\langle b\rangle$, and $\langle ab^k\rangle$
  where $k\geq 1$ is an integer. (Notice that any subgroup of order
  $p$ is generated by some element $ab^k$ for some integer $k$.) In
  other words, we have the relations $cac^{-1}=a^{\lambda_1}$,
  $cbc^{-1}=b^{\lambda_2}$, and $cab^{k}c^{-1}=(ab^k)^t=a^tb^{tk}$ for
  some integers $\lambda_1,\lambda_2,t$. On the other hand,
  $cab^{k}c^{-1}=a^{\lambda_1}b^{\lambda_2k}$ implying
  $\lambda_1\equiv\lambda_2\pmod{p}$ and hence we may take
  $\lambda:=\lambda_1=\lambda_2$. As a consequence all $p$-subgroups
  are normal in $G$. Notice that $\lambda=1$ implies
  $Q\triangleleft G$, hence $\lambda>1$. Moreover, since
  $a=c^qac^{-q}=a^{\lambda^q}$, we have $\lambda^q\equiv
  1\pmod{p}$.
  Conversely, it can be verified that a group with this presentation
  is of order $p^2q$.

  Let $G$ a the group with the given presentation. To conclude the
  proof, we shall show that $G$ is $3$-connected. We claim $G$
  satisfies $3$-valency condition. From the previous arguments,
  $\langle c\rangle$ is contained in at least three subgroups of order
  $pq$ and any element of order $q$ acts on $P$ in the same way as $c$
  does. Moreover, all $p$-subgroups are normal and there are clearly
  more than three proper subgroups containing any subgroup of order
  $p$. Finally, since the maximal subgroups of $G$ form a complete
  graph in $\Gamma(G)$ by the product formula, we deduce that $G$ is
  $3$-connected.
\end{proof}

\begin{proof}[Proof of Theorem B]
  Let $G$ be a finite solvable group which is not a $p$-group. (Finite
  $p$-groups that are not $2$-connected are presented in
  Lemma~\ref{lem:pgroup2}.) Since $G$ is a solvable group by
  assumption, there exists a maximal subgroup $M$ of $G$ of prime
  index. By Lemma~\ref{lem:valency2}, it is enough to determine groups
  for which the $2$-valency condition does not hold. Suppose that $G$
  does not satisfy $2$-valency condition. Let $A$ be a minimal
  subgroup of order $q$ such that $A$ is (strictly) contained in at
  most one proper subgroup of $G$.

  First, suppose that $q^2\divides |G|$ and let $Q$ be a Sylow
  $q$-subgroup containing $A$. Then either $M=Q$ and $|G|=pq^2$ where
  $[G:M]=p$, or $M\neq Q$ and $[G:M]=q$ as $Q$ is the unique subgroup
  containing $A$. The first case was considered in
  Lemma~\ref{lem:p2q}. In the latter case, if $Q\triangleleft G$ then $|G|=pq^2$ as $Q$ is the only proper subgroup containing $A$; hence, we again refer to Lemma~\ref{lem:p2q}. And if $Q\ntriangleleft G$, then we may further assume that $M$ is a normal subgroup. (Notice that since Sylow $q$-subgroup $Q$ is maximal, there must be a normal subgroup of $G$ of index $q$ in that case.) Moreover, $M$ must be a minimal normal subgroup as $Q$ is the unique proper subgroup containing $A$. Since $G$ is solvable $M\cong\mathbb{Z}_p\times\dots\times\mathbb{Z}_p$ for some prime $p$ different from $q$. However, this is impossible as $q\divides |M|$.

  Next, suppose that $A$ is a Sylow $q$-subgroup. If $|G|=pq$, clearly
  $2$-valency condition does not hold and the case $|G|=p^2q$ was
  already considered in Lemma~\ref{lem:p2q}. Suppose that
  $p,r\divides |G|$ where $p$ and $r$ are distinct prime numbers different
  from $q$. Since $G$ is solvable, there exist a Hall
  $\{p,q\}$-subgroup and a Hall $\{q,r\}$-subgroup containing
  $A$. Hence, we may assume $|G|=p^{\alpha}q$, $\alpha\geq 3$. If
  $A\triangleleft G$, then it is contained in more than one proper
  subgroup. Thus $A\ntriangleleft G$. Furthermore, $P\triangleleft G$ in this case, where $P$ is the Sylow $p$-subgroup of $G$. Suppose to the contrary that $P\ntriangleleft G$. Since $G$ is solvable, there exists a normal subgroup $M$ of index $p$. Hence, $M$ contains all Sylow $q$-subgroups implying $[M:N_M(A)]=[G:N_G(A)]$. On the other hand, since $M$ is the unique proper subgroup of $G$ containing $A$, either $N_M(A)=A$ or $N_M(A)=M$. In the first case $N_G(A)$ would be a subgroup of order $pq$ contradicting the assumption that $A$ is contained in at most one subgroup. And in the latter case $A$ would be a normal subgroup of $G$ which is again a contradiction. Therefore $P\triangleleft G$. For the rest of the proof, we take $Q:=A$. Let $N$ be a minimal normal subgroup of $G$. Since $G$ is solvable $N$ is elementary abelian. Now we examine two cases:

  \emph{Case I:} $P=N$. Thus, $P$ is elementary abelian. We claim that
  $G$ does \emph{not} satisfy $2$-valency condition if and only if $Q$
  acts on $P$ irreducibly ($Q$ does not normalize any proper non-trivial subgroup of $P$) and the order of $N_G(Q)$ is at
  most $pq$. Sufficiency is obvious. For the necessity, it is obvious that if $Q$ is a maximal subgroup of $G$, then $Q$ is self-normalizing and acts on $P$ irreducibly. If $Q$ is not maximal we argue as follows: Let $K$ be a minimal subgroup properly containing $Q$. Then either
  $Q$ is a normal subgroup of $K$ (hence $K\leq N_G(Q)$) or
  $O_p(K)$ is a non-trivial normal subgroup of $K$. In the first case, clearly $K=N_G(Q)$ and $|K|=pq$. Moreover, as $K$ is the unique proper subgroup of $G$ containing $Q$ by assumption, the action of $Q$ on $P$ is irreducible. In the latter case, since $O_p(K)$ is also a normal subgroup of $G$ which is not trivial, $P$ and $O_p(K)$ coincides by the minimality of $P$ implying $K=G$. The action of $Q$ on $P$ is irreducible in this case also.

  \emph{Case II:} $P\neq N$.

  \emph{Case II (a):} $P$ is elementary abelian. A Theorem of Gaschütz
  (see \cite[Theorem 7.43]{Rotman1995}) states that an abelian normal
  $p$-subgroup has a complement in $G$ if and only if it has a
  complement in a Sylow $p$-subgroup. Clearly, $N$ is complemented in
  the elementary abelian $p$-subgroup $P$. Let $K$ be the complement
  of $N$ in $G$. Then $NQ$ and $K$ are two distinct subgroups
  containing $A$.

  \emph{Case II (b):} $P$ is \emph{not} elementary abelian. As $NQ$ is
  a proper subgroup of $G$ containing $A$, we conclude that $P$ and
  $N$ are the only proper non-trivial normal subgroups of
  $G$. Moreover, $N$ coincides with $\Phi(P)$, since $\Phi(P)$ is a
  non-trivial characteristic subgroup of $M$. Notice that a
  characteristic subgroup of a normal subgroup is normal in the whole
  group (see \cite[Lemma 5.20]{Rotman1995}). That is, $P$ is a
  $p$-group such that its Frattini subgroup $N$ is elementary
  abelian. Moreover, since $Q$ is contained in at most one subgroup,
  either $N_G(Q)=Q$ or $N_G(Q)=NQ$. Notice that in
  the latter case we have $NQ\cong \mathbb{Z}_p\times
  \mathbb{Z}_q$.
  Consider the action of $Q$ on the set of subgroups of $P$ by
  conjugation. It is easy to see that the fixed points of this action
  must be precisely $P$, $N$, and the trivial subgroup. Clearly, $Q$
  acts on $N$ irreducibly and by the Correspondence Theorem the
  induced action of $Q$ on $P/N$ is also irreducible. Conversely, if
  the action of $Q$ on $N$ and $P/N$ are irreducible, then $N$ is the
  only proper non-trivial subgroup of $P$ fixed by $Q$. To see this,
  take an element $a\in Q$ and consider its action. If $1<X<N$, then
  $X^a\neq X$ by assumption. Let $N<NX,\,X\neq P$ and $X^a=Y$. Notice that since $N$ is contained by every maximal subgroup of $G$, we have $NX\neq P$. We want to
  show that $X\neq Y$. By assumption $(NX/N)^a\neq NX/N$. However,
  $(NX/N)^a=(NX)^a/N=NY/N$ implying $X\neq Y$.
\end{proof}

\begin{rem}
  The ``smallest'' non-solvable group is the alternating group $A_5$
  on five letters and its order is divisible by three distinct
  primes. However, it does not satisfy the $2$-valency condition. To be
  more precise, if $H$ is a subgroup of order $5$, then there is
  exactly one proper subgroup, say $K$, of $A_5$ containing $H$. To
  see this, first observe that any maximal subgroup $M$ of $A_5$ has
  index $\geq 5$, as otherwise, there would be a homomorphism
  $\phi\colon A_5\to S_{G/M}$ with a non-trivial kernel which is
  impossible. Hence the only possibility for the order of $K$ is
  $10$. Since $H$ is not a normal subgroup of $A_5$ and since $H$ is normalized by the maximal subgroup $K$, we see that $K$ is the unique subgroup containing $H$.
\end{rem}

\subsection*{Nilpotent groups}

As it was mentioned in the Introduction, to show that the intersection graph of a given solvable group is $3$-connected we must claim the existence of ``sufficiently'' many vertices to construct at least three independent paths for any pair of minimal subgroups which seems to be not an easy task. (Or, conversely, we must claim the non-existence of vertices to verify that the graph is not $3$-connected.) Of course, Hall Theorems enables us to claim that $3$-valency condition is satisfied if there are at least four distinct prime divisors of the order of the group. Also, it is not difficult to show that such groups are indeed $3$-connected (compare with Corollory~\ref{cor:divisors} below). However, if there are less than four prime divisors things are more complicated. Therefore, in this section we restrict our attention to nilpotent groups.   

\begin{lem}\label{lem:valency3}
  Let $G$ be a finite supersolvable group. Then $\kappa(G)\geq 3$ if and
  only if $G$ satisfies the $3$-valency condition.
\end{lem}

\begin{proof}
  Sufficiency is obvious. Let $G$ be a finite supersolvable group
  satisfying the $3$-valency condition. We want to show that there are
  at least three independent paths between any pair of minimal
  subgroups $A$ and $B$. Clearly, we may suppose that $\ell(G)\geq 3$,
  since groups of order $p^2$ and $pq$ does not satisfy even
  $1$-valency condition. Notice that as $G$ is supersolvable, any
  maximal subgroup is of prime index; and thus, if $X$ is a
  non-trivial subgroup of $G$ which is not minimal, then $X$
  intersects any maximal subgroup non-trivially.

  \emph{Case I:} $G$ has exactly one maximal subgroup. Then $G$ is a
  cyclic group of prime power order $p^{\alpha}$ and it satisfies
  $3$-valency condition if and only if $\alpha\geq 5$ which is the
  case if and only if $\kappa(G)\geq 3$.

  \emph{Case II:} $G$ has exactly two maximal subgroups. If $G$ is a
  $p$-group, then the number of maximal subgroups $\equiv 1 \pmod{p}$
  (see \cite[p.~30]{Berkovich2008}). Also, if $|G|$ is divisible by
  three distinct prime divisors, then it would have at least three
  maximal subgroups corresponding (containing the corresponding Hall
  subgroups). Hence $|G|=p^aq^b$. Obviously, maximal subgroups must be
  normal and hence $G\cong P\times Q$ is a nilpotent group where $P$ and
  $Q$ are Sylow $p$- and Sylow $q$- subgroups respectively. Observe
  that if $H\trianglelefteq P$ and $K\trianglelefteq Q$, then
  $HK\trianglelefteq G$ as $G$ is the direct product of $P$ and
  $Q$. However, since any maximal subgroup of a $p$-group is normal,
  $P$ and $Q$ have exactly one maximal subgroups meaning both are
  cyclic groups of prime power order and in turn $G$ is also a cyclic
  group. It can be easily observed that $3$-valency condition is
  equivalent to the $3$-connectedness for such groups.

  \emph{Case III:} $G$ has at least three maximal subgroups. Let $M_i$
  be maximal subgroups, $X_i$ be subgroups containing $A$, and $Y_i$
  be subgroups containing $B$ for $1\leq i\leq 3$. Then
  $(A,X_i,M_i,Y_i,B)$, $1\leq i \leq 3$, are three independent paths
  between $A$ and $B$. Of course, in case of a coincidence the
  corresponding paths can be shortened accordingly.
\end{proof}

\begin{cor}\label{cor:divisors}
  Let $G$ be a finite supersolvable group with $\kappa(G)<3$. Then the
  number of prime divisors of $|G|$ is at most three. Moreover, if
  there are exactly three distinct prime divisors, then $|G|$ is
  square-free.
\end{cor}

\begin{proof}
  Obviously, if there are more than three distinct prime divisors of
  $|G|$, then $G$ satisfies $3$-valency condition, hence is
  $3$-connected as well. Let $|G|=p^{\alpha}q^{\beta}r^{\gamma}$ where
  $p,q,r$ are distinct prime numbers and $\alpha\geq 2$. Let $A$ be a
  minimal subgroup. If $A$ is a $p$-subgroup, then $A$ is properly
  contained in a Sylow $p$-subgroup, by a Hall $\{p,q\}$-subgroup, and
  by a Hall $\{p,r\}$-subgroup. If $A$ is a $q$-subgroup, then $A$ is
  contained in Hall $\{p,q\}$-subgroup, by a Hall $\{q,r\}$-subgroup
  and by a maximal subgroup containing the corresponding Hall
  $\{q,r\}$-subgroup. Similarly, there are at least three proper
  subgroups containing $A$ whenever $A$ is a $r$-subgroup.
\end{proof}

By Corollory~\ref{cor:order}, we know that if $G$ is a supersolvable group such that $\kappa(G)<3$, then $\ell(G)$ is at most $5$. Moreover, by using Corollory~\ref{cor:divisors} (and ignoring the $\ell(G)\leq 2$ cases), we may reduce the possible cases for the order of $G$ into the following list

\begin{table}[h!]
\centering
\caption{}
\label{table:1}
\begin{tabular}{lll}
  $ \vert G \vert = p^5 $,     & $ \vert G \vert = p^4 $,    & $ \vert G \vert = p^3 $,    \\
  $ \vert G \vert = p^4q $,    & $ \vert G \vert = p^3q $,   & $ \vert G \vert = p^2q $,   \\
  $ \vert G \vert = p^3q^2 $,  & $ \vert G \vert = p^2q^2 $, & $ \vert G \vert = pqr $.    \\
\end{tabular}
\end{table}

Actually, we may still eliminate some further cases.

\begin{lem}\label{lem:order}
  Let $G$ be a finite supersolvable group with $\kappa(G)<3$. Then the
  order of $G$ must be equal to one of the following
  $$ p^{\alpha}\; (0\leq \alpha \leq 4), \quad p^3q, \quad p^2q^2,
  \quad p^2q,\quad pqr, \quad pq $$
  where $p,q,$ and $r$ are distinct prime numbers. Moreover, if $G$ is
  nilpotent, then $|G|\neq p^2q^2$; and if $G$ is nilpotent and of
  order $p^3q$, then $G$ is cyclic.
\end{lem}

\begin{proof}
  By Lemma~\ref{lem:valency3} we know that $G$ is $3$-connected, if
  $3$-valency condition holds. Since $G$ is supersolvable group, any
  minimal subgroup is contained in at least $\ell(G)-2$ proper
  subgroups. Thus, $\ell(G)< 5$. This eliminates the first column of
  Table~\ref{table:1}.
  
  Let $G$ be a nilpotent group of order $p^2q^2$ and let $A$ be a
  minimal subgroup of $G$. Obviously, $G$ is an abelian group and the
  normal subgroup $A$ of $G$ is contained in a subgroup of order $pq$,
  by a subgroup of order $p^2q$ and by the Sylow subgroup. Hence, $G$
  satisfies $3$-valency condition.
  
  Let $G$ be a nilpotent group of order $p^3q$. Then the Sylow
  $q$-subgroup $Q$ of $G$ is a normal subgroup and any minimal
  subgroup of order $p$ is contained in a subgroup of order $p^2$, by
  a subgroup of order $p^3$, and by a subgroup of order $pq$. Suppose
  that $Q$ is not contained in more than two proper
  subgroups. However, this is possible only if the Sylow $p$-subgroup
  $P$ is normal and $P\cong\mathbb{Z}_{p^3}$. Thus, $G$ is cyclic as
  well.
\end{proof}

\begin{lem}\label{lem:pgroup3}
  Let $G$ be a finite $p$-group. Then $\kappa(G)<3$ if and only if one
  of the following holds.
  \begin{enumerate}
  \item $|G|=p^{\alpha}$ $(0\leq\alpha\leq 3)$ and neither
    $G\cong Q_8$ nor
    $G\cong \mathbb{Z}_p\times\mathbb{Z}_p\times\mathbb{Z}_p$.
  \item $G$ is a group of order $p^4$ such that
    \begin{enumerate}
    \item $G\cong \mathbb{Z}_{p^4}$, or
    \item $\Phi(G)\cong\mathbb{Z}_{p^2}$ except $G\cong Q_{16}$, or
    \item
      $\Phi(G)\cong\mathbb{Z}_p\times\mathbb{Z}_p$ and $Z(G)<\Phi(G)$
      except
      \begin{enumerate}
      \item $G\cong \langle a,b,c\mid a^9=b^3=1, ab=ba, a^3=c^3,
      bcb^{-1}=c^4, aca^{-1}=cb^{-1} \rangle $, and 
      \item $G\cong \langle a,b,c\mid a^{p^2}=b^p=c^p=1, bc=cb, bab^{-1}=a^{p+1}, cac^{-1}=ab \rangle$ for $p>3$. 
      \end{enumerate}      
    \end{enumerate}
  \end{enumerate}
  In particular, all $p$-groups of order $>p^4$ are $3$-connected.
\end{lem}

\begin{proof}
  Obviously, $|G|=p^2$ implies $\Gamma(G)$ consists of isolated
  vertices, hence it cannot be connected. So, let's assume $|G|>p^2$.

  \emph{Case I:} $|G|=p^3$. First, suppose that $\Phi(G)\neq 1$. If
  all maximal subgroups of $G$ are cyclic, then $G$ has a unique
  minimal subgroup and its intersection graph is complete. However,
  $\Gamma(\mathbb{Z}_{p^3})$ has two vertices whereas $\Gamma(Q_8)$
  has four, thus only the latter is $3$-connected among them. If there
  exists a maximal subgroup $M\cong \mathbb{Z}_p\times\mathbb{Z}_p$,
  then any minimal subgroup $X$ of $M$ which is different from
  $\Phi(G)$ is not contained in any maximal subgroup other than $M$,
  as $\langle X,\Phi(G)\rangle$ uniquely determines $M$. That is, $G$
  does not satisfy $3$-valency condition in such case. Now, suppose
  that $\Phi(G)$ is trivial, i.e. $G$ is isomorphic to the elementary
  abelian group of rank $3$. By the correspondence theorem, any
  minimal subgroup is contained in $p+1$ maximal subgroups. Also,
  since any two maximal subgroups of a $p$-group intersects
  non-trivially (by the product formula), maximal subgroups form a
  complete subgraph in $\Gamma(G)$. Therefore, $G$ is $3$-connected in
  this case.

  \emph{Case II:} $|G|=p^4$. Recall that the \emph{rank} of a
  $p$-group is the dimension of $G/\Phi(G)$ as a vector space over the
  field of $p$-elements. If the rank of $G$ is four or three,
  i.e. $\Phi(G)\cong 1 \text{ or } \mathbb{Z}_p$, then for any minimal
  subgroup $X$ we may form $\Phi(G)X$ which is contained in at least
  $p+1$ maximal subgroups of $G$. Clearly, $G$ is $3$-connected in
  this case. On the other hand, if the rank of $G$ is one,
  i.e. $G\cong \mathbb{Z}_{p^4}$, then $\Gamma(G)$ has exactly three
  vertices and hence cannot be $3$-connected. Now we shall confine
  ourselves to the case $G$ is of rank two.

  Suppose that $\Phi(G)\cong\mathbb{Z}_{p^2}$. If $G$ has a unique
  minimal subgroup then it is isomorphic to the quaternion group
  $Q_{16}$ and its intersection graph is complete, hence $3$-connected
  as well. Let us assume there exists a minimal subgroup $X$ of $G$
  which is different from the minimal subgroup $P$ of
  $\Phi(G)$. Notice that $P$ is a necessarily normal subgroup of
  $G$. Then the only maximal subgroup containing $X$ is $M:=\Phi(G)X$
  as any maximal subgroup contains $\Phi(G)$. This in turn implies
  that $PX$ is the only subgroup of order $p^2$ containing $X$, since
  $P$ is the Frattini subgroup of $M$. Therefore $G$ does not satisfy
  $3$-valency condition in such a case, hence it is not $3$-connected.

  Suppose that $\Phi(G)\cong\mathbb{Z}_p\times\mathbb{Z}_p$. If $G$ is
  abelian, then it is isomorphic to
  $\mathbb{Z}_{p^2}\times\mathbb{Z}_{p^2}$ and any minimal subgroup is
  contained in the Frattini subgroup, hence it is $3$-connected. If
  $G$ is not abelian, then either $Z(G)=\Phi(G)$ or
  $Z(G)<\Phi(G)$. This is because, any cyclic extension of a
  central subgroup is abelian and $Z(G)$ intersects any
  normal subgroup non-trivially whenever $G$ is a $p$-group. Let
  $Z(G)=\Phi(G)$. Then a minimal subgroup $P$ is normal in
  $G$ if and only if $P$ is a subgroup of $\Phi(G)$. We show that $G$
  is $3$-connected in such a case. Let $P_i$, $1\leq i\leq p+1$ be
  minimal subgroups of $\Phi(G)$ and let $X,Y$ be two arbitrary
  minimal subgroups that are not contained in $\Phi(G)$. We show that
  there are at least three independent paths between any pair of
  minimal subgroups. Clearly, this holds if the endpoints are $P_i$
  and $P_j$ for any $i\neq j$. Let $A_i:=P_iX$ for $1\leq i\leq
  p+1$.
  Since $X\nleq \Phi(G)$ and $|A_i|=p^2$, $A_i\cap A_j=X$ for
  $i\neq j$. Then, we may form three internally independent paths
  $(X,A_1,M,P_i)$, $(X,A_2,N,P_i)$, and $(X,A_3,T,P_i)$ between $X$
  and $P_i$ where $M,N,$ and $T$ are mutually distinct maximal
  subgroups. Let $B_i:=P_iY$ for $1\leq i\leq p+1$. Clearly,
  $(X,A_i,B_i,Y)$, $1\leq i \leq p+1$ are independent paths between
  $X$ and $Y$.


Let $Z(G)<\Phi(G)$. Obviously, $Z:=Z(G)$ is the unique minimal normal subgroup of $G$. If all minimal subgroups of $G$ are contained in $\Phi(G)$, then clearly $3$-valency condition holds and $G$ is $3$-connected by Lemma~\ref{lem:valency3}. We want to show that under these conditions $G$ is unique up to isomorphism. 

  (I) There exists a maximal subgroup $A$ which is abelian, moreover
  $A\cong \mathbb{Z}_{p^2}\times\mathbb{Z}_p$. By the N/C lemma (see
  \cite[Theorem 7.1]{Rotman1995}),
  $N_G(\Phi(G))/C_G(\Phi(G))=G/C_G(\Phi(G))$
  can be embedded into
  $\mathrm{Aut}(\Phi(G))\cong \mathbb{Z}_p\times\mathbb{Z}_p$ which is
  of order $(p^2-1)(p^2-p)$. Then $C_G(\Phi(G))=\Phi(G)$
  implies $p^2\divides |\mathrm{Aut}(\Phi(G))|$ and this is
  impossible. Also, since the center of $G$ is a proper subgroup of
  $\Phi(G)$, then $C_G(\Phi(G))$ is not the whole group $G$
  either. Thus, $A:=C_G(\Phi(G))$ is an abelian subgroup of
  order $p^3$; and since any maximal subgroup $Y$ of $A$ different
  from $\Phi(G)$ is cyclic,
  $A\cong \mathbb{Z}_{p^2}\times\mathbb{Z}_p$.

  (II) $A:=C_G(\Phi(G))$ is the unique abelian group of order
  $p^3$, moreover $M\cong \mathbb{Z}_{p^2}\rtimes\mathbb{Z}_p$ for any
  maximal subgroup $M$ different from $A$. Suppose that there exists
  an abelian subgroup $B\cong \mathbb{Z}_{p^2}\times\mathbb{Z}_p$
  different from $A$. Then, as $\langle A,B\rangle=G$ and
  $A\cap B=\Phi(G)$, the center $Z$ of $G$ contains $\Phi(G)$ which is
  a contradiction. Therefore, any maximal subgroup $M$ other than $A$
  is isomorphic to $\mathbb{Z}_{p^2}\rtimes\mathbb{Z}_p$, since any
  non-Frattini subgroup of order $p^2$ is cyclic.

  (III) $G$ has a presentation
  $$ \langle a,b,c\mid a^{p^2}=b^p=1, ab=ba, a^p=c^{kp},
  bcb^{-1}=c^{1+p}, aca^{-1}=c^{1+mp}b^n \rangle$$
  for some suitable values of the prime $p$ and integers $k,m,n$. Let
  $a,b\in A$ and $c\in M$ such that $a,c$ are of order $p^2$ and $b$
  is of order $p$. Clearly, those elements generate $G$ and we have
  $ab=ba$, $a^p,c^p\in Z$, and $bcb^{-1}=c^{1+p}$ as
  $\langle b,c\rangle\cong
  \mathbb{Z}_{p^2}\rtimes\mathbb{Z}_p$.
  Moreover, any conjugate of $c$ can be written as
  $c^{\gamma}b^{\beta}$ for some integers $\gamma,\beta$; and since
  $ac^pa^{-1}=c^p=c^{rp}$ where $\gamma\equiv r \pmod{p}$, we have
  $r=1$.

  (IV) $p=3$. We want to show that for $p>3$, there is an element $g$
  of order $p$ such that $g\notin\Phi(G)$. Then the subgroup generated
  by this element is a minimal one and it is not contained in the
  Frattini subgroup contrary to our assumption. Thus, we shall deduce
  $p=3$. Using the above relations, we may obtain
  $b^{\beta}c^{\gamma}=c^{\gamma}b^{\beta}c^{\beta\gamma p}$ and
  $ac^x=c^xab^{xn}u$ where $u=c^{(\frac{1}{2}x(x-1)n+xm)p}$. Clearly
  $ac^x\notin\Phi(G)$ for $p\nmid x$. By some further computation
  $$
  (ac^x)^p=c^{xp\{1+\frac{1}{6}(p-1)p(p+1)xn\}}a^pb^{\frac{1}{2}p(p-1)xn}u^{\frac{1}{2}p(p-1)}=c^{xp\{1+\frac{1}{6}(p-1)p(p+1)xn\}}a^p.  $$
  However, this formula implies that $(ac^{-k})^p=1$ for $p\neq 3$.
 
  (V) Without loss of generality we may take $k=1, m=0, n=-1$. Suppose
  that $p=3$. Clearly, $k\in\{-1,1\}$ and $m,n\in\{-1,0,1\}$. However,
  $n=0$ implies that $C_G(\langle c\rangle)$ is an abelian
  group of order $p^3$ (compare with (I)). As
  $C_G(\langle c\rangle)$ is different from $A$, this
  contradicts with (II). Moreover, using the relation presented in
  (IV), we see that $(ac^k)^3=c^{3(n-k)}$. Therefore, $n$ and $k$ have
  opposite parity, as otherwise, $\langle ac^k\rangle$ would be a
  minimal subgroup which is not contained in $\Phi(G)$. Thus, there
  are totally six distinct triples $(k,m,n)$ that we shall
  consider. If triples $(k_1,m_1,n_1)$ and $(k_2,m_2,n_2)$ yields
  isomorphic groups, we simply write
  $(k_1,m_1,n_1)\sim(k_2,m_2,n_2)$. Now substituting $a^{-1}$ for $a$
  yields an automorphism of $G$ showing that $(1,0,-1)\sim(-1,0,1)$,\,
  $(1,-1,-1)\sim(-1,1,1)$, and $(1,1,-1)\sim(-1,-1,1)$. Also, it can
  be verified that the automorphisms
  $\varphi\colon a\mapsto a, b\mapsto b, c\mapsto c^2$ and
  $\psi\colon a\mapsto ab, b\mapsto b, c\mapsto b^{-1}cb$ yields
  $(1,0,-1)\sim(-1,1,1)$ and $(1,0,-1)\sim(1,1,-1)$
  respectively. Hence, we have
  $$
  (1,0,-1)\sim(1,-1,-1)\sim(1,1,-1)\sim(-1,0,1)\sim(-1,1,1)\sim(-1,-1,1) $$
  Conversely, it can be verified that a group with this presentation
  is of order $81$ and all minimal subgroups are contained in
  $\Phi(G)$. 


Next suppose that there are minimal subgroups of $G$ which do not lie in $\Phi(G)$. Let $X$ be such a minimal subgroup. Observe that $X$ is contained in exactly one maximal subgroup, say $M$. Otherwise, there exist two distinct maximal subgroups such that their intersection strictly contains $\Phi(G)$ which is impossible. For the rest of the proof we assume that $G$ is a $3$-connected group. Then, there exist two subgroups $Y$ and $W$ of order $p^2$ containing $X$, as otherwise,  $3$-valency condition would not be satisfied. Since $Y,W,$ and $\Phi(G)$ are maximal subgroups of $M$ intersecting trivially, we see that $M\cong \mathbb{Z}_p\times\mathbb{Z}_p\times\mathbb{Z}_p$. Moreover, $M$ contains all minimal subgroups of $G$, as otherwise, there would be another maximal subgroup of $G$ which is abelian and this implies that $\Phi(G)$ is in the center of $G$ contrary to our assumption. In other words, any maximal subgroup $L$ of $G$ different from $M$ is isomorphic to $\mathbb{Z}_{p^2}\rtimes\mathbb{Z}_p$. We want to show that $G$ has a presentation 
$$ G\cong \langle a,b,c\mid a^{p^2}=b^p=c^p=1, bc=cb, bab^{-1}=a^{p+1}, cac^{-1}=ab \rangle.  $$
Let $a,b\in L$ and $c\in M$ such that $a$ is of order $p^2$ and $b,c$ are of order $p$. Clearly, those elements generate $G$ and we have $bc=cb$, $a^p\in Z$, and $bab^{-1}=a^{p+1}$ as $\langle a,b\rangle\cong\mathbb{Z}_{p^2}\times\mathbb{Z}_p$. Moreover, any conjugate of $a$ can be written as $a^{\alpha}b^{\beta}$ for some integers $\alpha,\beta$; and since $ca^pc^{-1}=a^p=a^{rp}$ where $\gamma\equiv r \pmod{p}$, we have $r=1$. One may observe that distinct conjugates of $\langle a \rangle$ are generated by $a^{1+mp}b$ for $1\leq m\leq p$; and $c\in M$ can be chosen in a way that $cac^{-1}=ab$ holds. Conversely, it can be verified that a group with this presentation is of order $p^4$ and has the desired structure for $p>3$. Notice that the generators $a$ and $b$ commutes for $p=2$. Also the element $ac$ clearly does not belong to $\Phi(G)$; however its order is $p$ for $p=3$. For the classification of groups of order $p^4$, the reader may refer to \cite[p.~140]{Burnside1955}.

\end{proof}

\begin{proof}[Proof of Theorem C]
  It can be easily verified that nilpotent groups of order $p^2q$ do
  not satisfy the $3$-valency condition. (This is also a consequence
  of Lemma~\ref{lem:p2q-2}.) Also, a nilpotent group of order $pqr$ is
  cyclic and does not satisfy the $3$-valency condition. Then the first part of the Theorem follows from Lemmas \ref{lem:valency3}, \ref{lem:pgroup3}, and
  \ref{lem:order}. 

For the second part we argue as follows. Let $A$ and $B$ be two distinct minimal subgroups of a finite solvable group $G$ such that there are at least four distinct prime divisors of $|G|$. Suppose that $A$ and $B$ are of same order, say $p$. Let $A_q$, $A_r$, and $A_s$ be some maximal Hall subgroups of $G$ containing $A$ such that their indexes is a power of prime numbers $q$, $r$, and $s$ respectively. Also, let $B_q$, $B_r$, and $B_s$ be some maximal Hall subgroups containing $B$. (Of course, $[G:B_q]=q^{\alpha}$ for some integer $\alpha$, and so on.) By the product formula $(A,A_q,B_r,B)$, $(A,A_r,B_s,B)$, and $(A,A_s,B_q,B)$ are three independent paths between $A$ and $B$. Similar arguments can be applied if $|A|\neq|B|$.
\end{proof}


\end{document}